\documentclass[graybox]{svmult}

\usepackage{type1cm}        
\usepackage{makeidx}         
\usepackage{graphicx}        
\usepackage{multicol}        
\usepackage[bottom]{footmisc}

\usepackage{newtxtext}       %
\usepackage{newtxmath}       

\makeindex             

\newcommand{\MG}[1]{\textcolor{black}{#1}}

\begin{document}

\title*{Optimizing transmission conditions for \MG{multiple} subdomains
  in the Magnetotelluric Approximation of Maxwell's equations}
\titlerunning{Optimizing transmission conditions for many subdomains}

\author{V. Dolean, M.J. Gander and A. Kyriakis}

\institute{Alexandros Kyriakis \at University of Strathclyde,  \email{alexandros.kyriakis@strath.ac.uk}
\and Victorita Dolean \at University of Strathclyde and University Côte d'Azur  \email{work@victoritadolean.com} 
\and Martin J. Gander \at Université de Genève \email{martin.gander@unige.ch}}
 
%
%
\maketitle

\abstract{Classically transmission conditions between subdomains are
  optimized for a simplified two subdomain decomposition to obtain
  optimized Schwarz methods for many subdomains. We investigate here
  if such a simplified optimization suffices for the magnetotelluric
  approximation of Maxwell's equation which leads to a complex
  diffusion problem.  We start with a direct analysis for 2 and 3
  subdomains, and present asymptotically optimized transmission
  conditions in each case. We then optimize transmission conditions
  numerically for 4, 5 and 6 subdomains and observe the same
  asymptotic behavior of optimized transmission conditions. We finally
  use the technique of limiting spectra to optimize for a very large
  number of subdomains in a strip decomposition. Our analysis shows
  that the asymptotically best choice of transmission conditions is
  the same in all these situations, only the constants differ
  slightly. It is therefore enough for such diffusive type
  approximations of Maxwell's equations, which include the special case
  of the Laplace and screened Laplace equation, to optimize
  transmission parameters in the simplified two subdomain
  decomposition setting to obtain good transmission conditions for
  optimized Schwarz methods for more general decompositions.}

\section{Optimized Schwarz for the Magnetotelluric Approximation}\label{sec:1}

Wave propagation phenomena are ubiquitous in science and
engineering. In Geophysics, the magnetotelluric approximation of
Maxwell's equations is an important tool to extract information about
the spatial variation of electrical conductivity in the Earth's
subsurface. This approximation results in a complex diffusion equation
\cite{Donzelli:2019:ASM},
\begin{equation}\label{eq:1}
  \Delta u - (\sigma - i\varepsilon )u = f, \quad \mbox{in a domain $\Omega$},
\end{equation}
where $f$ is the source function, and $\sigma$ and $\varepsilon$ are
strictly positive constants\footnote{In the magnetotelluric
  approximation we have $\sigma=0$, but we consider the slightly more
  general case here. Note also that the zeroth order term in
  \eqref{eq:1} is much more benign than the zeroth order term of
  opposite sign in the Helmholtz equation, see e.g. \cite{Ernst:2012:WDS}.}.

To study Optimized Schwarz Methods (OSMs) for \eqref{eq:1}, we use \MG{a}
rectangular domain $\Omega$ given by the union of rectangular
  subdomains $\Omega_j:= ({a_j},{b_j}) \times (0,\hat L)$,
  $j=1,2,\ldots,J$, where $a_j=(j-1)L-\frac{\delta}{2}$ and
  $b_j=jL+\frac{\delta}{2}$, and $\delta$ is the overlap, like in
  \cite{Chaouqui:2018:OSC}. Our OSM computes for iteration index
  $n=1,2,\hdots$
\begin{equation} \label{eq:2}
  \begin{array}{rcll}
  \Delta u_j^{n} -(\sigma-i \varepsilon) u_j^{n}&=&f&\mbox{in $\Omega_j$},\\
  -\partial_x u_j^{n}+p_j^{-} u_j^{n}&=&-\partial_x u_{j-1}^{n-1}+p_j^{-}u_{j-1}^{n-1}
  \quad &\mbox{at $x=a_j$}, \\
  \partial_x u_j^{n}+p_j^{+} u_j^{n}&=&\partial_x u_{j+1}^{n-1}+p_j^{+} u_{j+1}^{n-1}
  \quad &\mbox{at $x=b_j$},
\end{array} 
\end{equation}
where $p_j^{-}$ and $p_j^{+}$ are strictly positive parameters in the
so called 2-sided OSM, see e.g. \cite{Gander:2007:OSM}, and we have at
the top and bottom homogeneous Dirichlet boundary conditions, and on
the left and right homogeneous Robin boundary conditions, i.e we put
for simplicity of notation $u_0^{n-1}=u_{J+1}^{n-1}=0$ in
\eqref{eq:2}. Note that the parameters $p_{j}^{-}$,$p_{j}^{+}$ are real and not complex (as one would expect in the case of a complex problem) for the sake of simplicity in our analysis. The Robin parameters are fixed at the domain boundaries
$x=a_1$ and $x=b_J$ to $p_1^-= p_a$ and $p_J^-+= p_b$. As $p_a$, $p_b$ tend to infinity, this is equivalent to imposing Dirichlet conditions. By linearity,
it suffices to study the homogeneous equations, $f=0$, and analyze
convergence to zero of the OSM \eqref{eq:2}.  Expanding the
homogeneous iterates in a Fourier series
$u_j^{n}(x,y)=\sum_{m=1}^{\infty}
v_j^{n}(x,\tilde{k})\sin(\tilde{k}y)$ where
$\tilde{k}=\frac{m\pi}{\hat{L}}$ to satisfy the homogeneous Dirichlet
boundary conditions at the top and bottom, we obtain for the Fourier
coefficients the equations
\begin{equation}
\label{eq:3}
  \begin{array}{rcll}
  \partial_{xx}v_j^{n}-(\tilde{k}^{2}+\sigma-i\varepsilon)
  v_j^{n}&=&0 & x\in  (a_j,b_j),  \\
  -\partial_x v_j^{n}+p_j^{-} v_j^{n}&=&
  -\partial_x v_{j-1}^{n-1}+p_j^{-} v_{j-1}^{n-1}\quad
  & \mbox{at $x=a_j$},     \\
  \partial_x v_j^{n}+p_j^{+} v_j^{n}&=&
  \partial_x v_{j+1}^{n-1}+p_j^{+} v_{j+1}^{n-1}\quad &
  \mbox{at $x=b_j$}.
\end{array} 
\end{equation}
The general solution of the differential equation is $
  v_j^{n}(x,\tilde{k})=\tilde{c}_j e^{-\lambda(\tilde{k})x}
    +\tilde{d}_j e^{\lambda(\tilde{k})x},$  where
$\lambda=\lambda(\tilde{k})=\sqrt{\tilde{k}^{2}+\sigma-i\varepsilon}$.
We next define the Robin traces,
$\mathcal{R}_{-}^{n-1}(a_j,\tilde{k}):= -\partial_x
v_{j-1}^{n-1}(a_{j},\tilde{k})+p_j^{-}v_{j-1}^{n-1}(a_{j},\tilde{k})$
and $\mathcal{R}_{+}^{n-1}(b_j,\tilde{k}):= \partial_x
v_{j+1}^{n-1}(b_j,\tilde{k})+p_j^{+}v_{j+1}^{n-1}(b_j,\tilde{k})$. Inserting
the solution into the transmission conditions in \eqref{eq:3},  a linear system arises where the unknowns  are $\tilde{c}_j $ and $\tilde{d}_j $ , whose solution is
\begin{align*}
  \tilde{c}_j &= \frac{1}{D_j}({e^{\lambda{b_j}}}(p_j^{+} + \lambda )
  \mathcal{R}_{-}^{n-1}(a_j,\tilde{k})
  -e^{\lambda{a_j}}(p_j^{-} - \lambda ){\mathcal{R}_{+}^{n-1}(b_j,\tilde{k})}),
  \\
  \tilde{d}_j &= \frac{1}{D_j}( - {e^{ - \lambda {b_j}}}(p_j^{+}- \lambda)
  \mathcal{R}_{-}^{n-1}(a_j,\tilde{k}) + {e^{ - \lambda {a_j}}}(p_j^{-} + \lambda )
  \mathcal{R}_{+}^{n-1}(b_j,\tilde{k})), 
\end{align*}
where $D_j:=(\lambda+p_j^{+})(\lambda+p_j^{-})e^{\lambda (L + \delta )}
  - (\lambda-p_j^{+})(\lambda-p_j^{-})e^{ - \lambda (L + \delta )}$.
We thus arrive for the Robin traces in the OSM at the iteration formula
\begin{align*}
{\mathcal{R}_{-}^{n}(a_j,\tilde{k})} 
&= \alpha_j^-\mathcal{R}_{-}^{n - 1}({a_{j-1}},\tilde{k})
+\beta_j^-\mathcal{R}_{+}^{n - 1}({b_{j-1}},\tilde{k}), \, j=2,\MG{\ldots},J\MG{,}\\
{\mathcal{R}_{+}^{n}(b_j,\tilde{k})} 
 &= \beta_j^+\mathcal{R}_{-}^{n - 1}({a_{j+1}},\tilde{k})
 +\alpha_j^+ \mathcal{R}_{+}^{n - 1}({b_{j+1}},\tilde{k}),  \, j=1,\MG{\ldots},J-1, 
\end{align*}
where 
\begin{align*}
\alpha_j^-\MG{:=}\frac{(\lambda+p_{j-1}^{+})(\lambda + p_j^{-})e^{\lambda \delta}-(\lambda-p_{j-1}^{+})(\lambda-p_j^{-})e^{-\lambda \delta}}
{(\lambda\!+\!p_{j-1}^{+})(\lambda\!+\!p_{j-1}^{-})e^{\lambda (L + \delta )}\!-\! (\lambda\!-\!p_{j-1}^{+})(\lambda\!-\!p_{j-1}^{-})e^{ - \lambda (L + \delta )}},\, j=2,\MG{\ldots},J\MG{,} \\
\alpha_j^+ \MG{:=}\frac{(\lambda+p_{j+1}^{-})(\lambda + p_j^{+})e^{\lambda \delta}-(\lambda-p_{j+1}^{-})(\lambda-p_j^{+})e^{-\lambda \delta}}
{(\lambda\!+\!p_{j+1}^{+})(\lambda\!+\!p_{j+1}^{-})e^{\lambda (L + \delta )}\!-\! (\lambda\!-\!p_{j+1}^{+})(\lambda\!-\!p_{j+1}^{-})e^{ - \lambda (L + \delta )}},\, j=1,\MG{\ldots},J-1\MG{,}
\end{align*}
\begin{align*}
\beta_j^{-}:=\frac{(\lambda + p_j^-)(\lambda - p_{j-1}^-) e^{-\lambda L}-(\lambda -p_j^-)(\lambda + p_{j-1}^-) e^{\lambda L}}
{(\lambda\!+\!p_{j-1}^{+})(\lambda\!+\!p_{j-1}^{-})e^{\lambda (L + \delta )}\!-\! (\lambda\!-\!p_{j-1}^{+})(\lambda\!-\!p_{j-1}^{-})e^{ - \lambda (L + \delta )}},\, j=2,\MG{\ldots},J,\\
\beta_j^{+}:=\frac{(\lambda + p_j^+)(\lambda - p_{j+1}^+) e^{-\lambda L}-(\lambda -p_j^+)(\lambda + p_{j+1}^+) e^{\lambda L}}
{(\lambda\!+\!p_{j+1}^{+})(\lambda\!+\!p_{j+1}^{-})e^{\lambda (L + \delta )}\!-\! (\lambda\!-\!p_{j+1}^{+})(\lambda\!-\!p_{j+1}^{-})e^{ - \lambda (L + \delta )}},\, j=1,\MG{\ldots},J-1\MG{.}
\end{align*}
Defining the matrices
$$
  T_j^1:=\left[ {\begin{array}{cc}
    \alpha_j^-&\beta_j^{-}\\
    0&0
  \end{array}} \right], \, j=2,..,J \quad \mbox{and}\quad 
  T_j^2:=\left[ {\begin{array}{cc}
  0&0\\
\beta_j^{+}&\alpha_j^+
\end{array}} \right],\, j=1,..,J-1\MG{,}
$$
we can write the OSM in substructured form \MG{(keeping the
  first and last rows and columns to make the block structure
  appear), namely}
\begin{equation}\label{SubstructuredForm}
\underbrace{\left[ {\begin{array}{c}
0\\
\mathcal{R}_{+}^n({b_1},\tilde{k})\\
\mathcal{R}_{-}^n({a_2},\tilde{k})\\
\mathcal{R}_{+}^n({b_2},\tilde{k})\\
 \vdots \\
\mathcal{R}_{-}^n({a_j},\tilde{k})\\
\mathcal{R}_{+}^n({b_j},\tilde{k})\\
 \vdots \\
\mathcal{R}_{-}^n({a_{N - 1}},\tilde{k})\\
\mathcal{R}_{+}^n({b_{N - 1}},\tilde{k})\\
\mathcal{R}_{-}^n({a_N},\tilde{k})\\
0
    \end{array}} \right] }_\text{${{{\cal R} }^n}$}=\underbrace{ \left[\arraycolsep0.5em {\begin{array}{ccccccc}
        \\[-0.2em]
 &T_1^2& & & & &  \\[0.7em]
T_2^1& &T_2^2& & & &  \\[0.7em]
 &\ddots& &\ddots& & &  \\[0.7em]
 & &T_j^1& &T_j^2& &  \\[0.7em]
 & & &\ddots& &\ddots&  \\[0.7em]
 & & & &T_{N-1}^1& &T_{N-1}^2 \\[0.7em]
 & & & & &T_N^1& \\[1em]
\end{array}} \right]}_\text{$T$}  \underbrace{\left[ {\begin{array}{*{20}{c}}
0\\
\mathcal{R}_{+}^{n - 1}({b_1},\tilde{k})\\
\mathcal{R}_{-}^{n - 1}({a_2},\tilde{k})\\
\mathcal{R}_{+}^{n - 1}({b_2},\tilde{k})\\
 \vdots \\
\mathcal{R}_{-}^{n - 1}({a_j},\tilde{k})\\
\mathcal{R}_{+}^{n - 1}({b_j},\tilde{k})\\
 \vdots \\
\mathcal{R}_{-}^{n - 1}({a_{N - 1}},\tilde{k})\\
\mathcal{R}_{+}^{n - 1}({b_{N - 1}},\tilde{k})\\
\mathcal{R}_{-}^{n - 1}({a_N},\tilde{k})\\
0
    \end{array}} \right]}_\text{${{{\cal R} }^{n-1}}$}.
\end{equation}
If the parameters $p_j^{\pm}$ are constant over all the interfaces,
\MG{and we eliminate} the first and the last row \MG{and} column of
$T$, \MG{$T$} becomes a block Toeplitz matrix. \MG{The best choice of
  the parameters minimizes the spectral radius $\rho(T)$ over a
  numerically relevant range of frequencies
  $K:=[\tilde{k}_{\min},\tilde{k}_{\max}]$ with
  $\tilde{k}_{\min}:=\frac{\pi}{\hat{L}}$ (or $0$ for simplicity) and
  $\tilde{k}_{\max}:=\frac{M\pi}{\hat{L}}$, $M\sim\frac{1}{h}$, where $h$
  is the mesh size, and is thus solution of the min-max problem
  $\min_{p_j^{\pm}}\max_{\tilde{k}\in
    K}|\rho(T(\tilde{k},p_j^{\pm}))|.$ }

The traditional approach to obtain optimized transmission conditions
for optimized Schwarz methods is to optimize performance for a simple
two subdomain model problem, and then to use the result also in the
case of many subdomains. We want to study here if this approach is
justified, by directly optimizing the performance for two and more
subdomains, and then comparing the results. We obtain
  our results from insight by numerical optimisation for small
  overlap, in order to find asymptotic formulas for the convergence
  factor and the parameters involved. The constants in the asymptotic
  results are then obtained by rigorous analytical computations of
  asymptotic series. We thus do not obtain existence and uniqueness
  results, but our asymptotically optimized convergence factors
  equioscillate as one would expect. For Robin conditions with complex
  parameters for two subdomains, existence and uniqueness results can
  be found in B. Delourme and L. Halpern \cite{hal-02554807}.

\section{Optimization for 2, 3, 4, \MG{5 and 6} subdomains}

For two subdomains, the general substructured iteration matrix becomes
$$
T= \left[\begin{array}{cc}
0 & \beta_1^+ \\
\beta_2^- & 0 
\end{array}\right].
$$
The eigenvalues of this matrix are $\pm \sqrt{ \beta_1^+ \beta_2^-}$ and thus the square of the
convergence factor is $\rho^2 = \left| \beta_1^+ \beta_2^- \right|$. 

\begin{theorem}[\MG{Two Subdomain Optimization}]
  Let $s\MG{:=}\sqrt{\sigma - i\varepsilon}$, where the complex square
  root is taken with a positive real part, and let $C$ be the real
  constant
  \label{Theorem2sub}
  \begin{equation}
  \label{eq:C}
    C \MG{:=} \Re \frac{s ((p_b + s)(p_a + s)-(s - p_b)(s - p_a)
      e^{-4sL} )}{((s - p_a)e^{-2sL}+ s + p_a)((s - p_b)e^{-2sL} + s +
      p_b)}.
  \end{equation}
  where $p_a$ and $p_b$ are the Robin parameters at the outer boundaries.
  \MG{Then for} two subdomains with $p_1^{+}=p_2^{-}\MG{=:}p$ and
  \MG{$\tilde{k}_{\min}=0$}, the asymptotically optimized parameter
  $p$ for small overlap $\delta$ and \MG{associated} convergence
  factor are
  \begin{equation}\label{eq:p2dom}
    p = 2^{-1/3}C^{2/3} \delta^{-1/3}, \quad
    \rho = 1 - 2 \cdot 2^{1/3}C^{1/3} \delta^{1/3} + {\cal O}(\delta^{2/3}).
  \end{equation} 
  If $p_1^{+}\ne p_2^{-}$ \MG{and} $\tilde{k}_{min}=0$, the asymptotically
  optimized parameters for small overlap $\delta$ and associated
  convergence factor are
  \begin{equation}\label{eq:2p2dom}
    p_1^+=  2^{-2/5} C^{2/5} \delta^{-3/5},\, p_2^-=  2^{-4/5} C^{4/5} \delta^{-1/5},\, \rho= 1- 2 \cdot 2^{-1/5}C^{1/5}  \delta^{1/5} + {\cal O}(\delta^{2/5}).
\end{equation} 
\end{theorem}
\begin{proof}
\MG{From numerical experiments, we obtain that the solution of the
  min-max problem equioscillates, $\rho(0) = \rho(\tilde{k}^*)$, where
  $\tilde{k}^*$ is an interior maximum point, and asymptotically $p =
  C_p \delta^{-1/3}$, $\rho = 1 - C_R \delta^{1/3} + {\cal
    O}(\delta^{2/3})$, and} $\MG{\tilde{k}}^* = C_k \delta^{-2/3}$. By
\MG{expanding} for $\delta$ small, \MG{and setting the leading term in
  the derivative} $\frac{\partial \rho}{\partial \tilde{k}}
(\tilde{k}^*)$ to zero, we get $C_p = \frac{C_k^2}{2}$. \MG{Expanding
  the maximum leads to $\rho(\tilde{k}^*)= \rho(C_k \delta^{-2/3}) = 1
  - 2 C_k \delta^{1/3} +{\cal O}(\delta^{2/3})$,} therefore $C_R =
2C_k$. Finally the solution \MG{of the equioscillation equation
  $\rho(0) = \rho(\MG{\tilde{k}}^*)$ determines uniquely $C_k= 2^{1/3}C^{1/3}$}.

In the case with two parameters, \MG{we have two equioscillations,
  $\rho(0)=\rho(\tilde{k}_1^*) = \rho(\tilde{k}_2^*)$, where
  $\tilde{k}_j^*$ are two interior local maxima, and} asymptotically
$p_1 = C_{p1} \delta^{-3/5}$, $p_1 = C_{p1} \delta^{-1/5}$, $\rho = 1
- C_R \delta^{1/5} + {\cal O}(\delta^{2/5})$, $\tilde{k}_1^* = C_{k1}
\delta^{-2/5}$ and $\tilde{k}_2^* = C_{k2} \delta^{-4/5}$. By
\MG{expanding for $\delta$ small, and setting the leading terms in the
  derivatives} $\frac{\partial \rho}{\partial k}
(\MG{\tilde{k}}_{1,2}^*)$ to zero, and we get $C_{p1} = {C_{k2}^2}$,
$C_{p2}=\frac{C_{k1}^2}{C_{k2}^2}$. \MG{Expanding the maxima} leads to
$\rho(\MG{\tilde{k}}_1^*)= \rho(C_k \delta^{-2/5}) = 1 - 2
\frac{C_{k1}}{C_{k2}^2} \delta^{1/5} +{\cal O}(\delta^{2/5})$ and
$\rho(\MG{\tilde{k}}_2^*)= \rho(C_k \delta^{-4/5}) = 1 - 2 C_{k_2}
\delta^{1/5} +{\cal O}(\delta^{2/5})$ and equating
$\rho(\MG{\tilde{k}}_1^*) = \rho(\MG{\tilde{k}}_2^*)$ we get $C_{k1} =
C_{k2}^3$ and $C_R = 2C_{k2}$.  Finally \MG{equating $\rho(0) =
  \rho(\MG{\tilde{k}}_2^*)$ asymptotically determines} uniquely
$C_{k2} = 2^{-1/5} C^{1/5}$ and then $C_{k1} = C_{k2}^3$ and $C_{p1}=
C_{k2}^2$, $C_{p2}= C_{k2}^4$.

\end{proof}

\begin{corollary}[Two Subdomains with Dirichlet \MG{outer} boundary conditions]
\MG{The case of Dirichlet outer boundary conditions can be obtained by
  letting $p_a$ and $p_b$ go to infinity, which simplifies}
  \eqref{eq:C} to
\begin{equation}
\label{eq:condDir}
C = \Re\frac{s(1+e^{2sL}) } {(e^{2sL}-1 )}
\end{equation}
and the \MG{asymptotic results in Theorem \ref{Theorem2sub} simplify}
accordingly.
\end{corollary}

For three subdomains, the general substructured iteration matrix becomes
$$
T= \left[ \begin{array}{cccc}
0&\beta_1^+&\alpha_1^+&0\\
\beta_2^-&0&0&0\\
0&0&0&\beta_2^+\\
0&\alpha_3^-&\beta_3^-&0
\end{array} \right],
$$
\MG{and we obtain for the first time an optimization result for three
  subdomains:}
\begin{theorem}[\MG{Three Subdomain Optimization}]
For three subdomains with \MG{equal parameters}
$p_1^{+}=p_2^{-}=p_2^{+}=p_3^{-}=p$\MG{, the} asymptotically optimized
parameter $p$ for small overlap $\delta$ and \MG{associated}
convergence factor are
\label{Theorem3sub}
\begin{equation}\label{eq:p3dom}
  p = 2^{-1/3}C^{2/3} \delta^{-1/3}, \quad \rho = 1 - 2 \cdot 2^{1/3}C^{1/3} \delta^{1/3} + {\cal O}(\delta^{2/3})\MG{,}
\end{equation} 
where $C$ is \MG{a real constant that can be obtained in closed
  form.}  If the parameters are different, \MG{their} asymptotically
optimized \MG{values} for small overlap $\delta$ are such that
\begin{equation}\label{eq:2p3dom}
    p_1^+, \, p_2^+ ,\,p_2^-,\,p_3^-\in \{2^{-2/5} C^{2/5}
    \delta^{-3/5},\, 2^{-4/5} C^{4/5} \delta^{-1/5}\},\ p_1^+\ne p_2^-,\
    p_2^+ \ne p_3^-,
  \end{equation} 
 and the associated convergence factor is
  \begin{equation}
  \label{eq:rho}
  \rho=  1- 2 \cdot 2^{-1/5}C^{1/5}  \delta^{1/5} + {\cal O}(\delta^{2/5}).
\end{equation}
\end{theorem}
\begin{proof}
The characteristic polynomial of the iteration matrix is
\[
G(\mu ) = {\mu ^4} - ( {\beta_{2, - }}{\beta_{1, + }} +{\beta_{3, - }}{\beta_{2, +
}}){\mu ^2} - {\alpha_3}{\beta_{2, - }}{\alpha_1}{\beta_{2, + }} + {\beta_{3, - }}{\beta_{2, +
}}{\beta_{2, - }}{\beta_{1, + }}.
\]
This biquadratic equation has the roots $
{\mu _1} = \pm\sqrt {\frac{{{m_1} + \sqrt {{m_2}} } }{2}}, 
{\mu _2} = \pm\sqrt {\frac{{\sqrt {{m_1} - \sqrt {{m_2}} } }}{2}}$
where
\[
  \begin{array}{l}
    {m_1} = {\beta_{2, - }}{\beta_{1, + }} + {\beta_{3, - }}{\beta_{2, + }},\, {m_2}= 4{\alpha_3}{\beta_{2, - }}{\alpha_1}{\beta_{2, + }} + ({\beta_{2, - }}{\beta_{1, + }} - {\beta_{3, - }}{\beta_{2, + }})^2\MG{.}
  \end{array}
\]
Therefore $\rho (T) = \max \{|\mu _1|,|\mu _2|\}$. Following the same
reasoning as in the proof of Theorem \ref{Theorem2sub}, we observe
that the solution equioscillates, and minimizing the
maximum asymptotically for $\delta$ small then leads to the desired
result, for more details, see \cite{ThesisAlex}.
\end{proof}
\MG{Notice that the optimized parameters and the relation between them is
the same as in the two-subdomain case, the only difference is the
equation whose solution gives the exact value of the constant $C$. The
only difference between a two subdomain optimization and a three
subdomain optimization is therefore the constant.
}
\begin{corollary}[Three subdomains with Dirichlet \MG{outer} boundary conditions]
  When Dirichlet boundary conditions are used at the end of the
  computational domain, we obtain \MG{for the constant}
\begin{equation}\label{eq:condDir3}
  C = \Re\frac{s (e^{2sL} - e^{sL} +1)}{e^{2sL}-1}\MG{,}
\end{equation}
  which is different from the two subdomain constant in
  \eqref{eq:condDir}.
\end{corollary}

For \MG{four subdomains, we show in Table \ref{Table:4sub}}
\begin{table}[t]
  \centering
  \caption{Asymptotic results for four subdomains: $\sigma = \varepsilon = 1, L=1, p_a = p_b = 1$}
  \label{Table:4sub}
  \begin{tabular}{c|ccccccc|cc}
    \hline\noalign{\smallskip}
		& \multicolumn{7}{c|}{Many parameters} &  \multicolumn{2}{c}{One parameter}  \\
		\noalign{\smallskip}\svhline\noalign{\smallskip}
		$\delta$ & $\rho$ & $p_1^+$ & $p_2^-$ & $p_2^+$ & $p_3^-$ & $p_3^+$ & $p_4^-$ & $\rho$ & $p$ \\
		\noalign{\smallskip}\svhline\noalign{\smallskip}
		$1/10^2$ & 0.5206 & 13.1269  & 1.2705 & 10.1871& 0.7748 & 16.5975  & 2.1327 & 0.6202 & 2.8396 \\
		$1/10^3$ & 0.6708 & 37.9717 & 1.4208 & 42.9379 & 1.6005 & 68.1923 & 2.4896  & 0.8022 & 6.0657\\
		$1/10^4$ & 0.7789 & 152.9323 & 2.3266 & 152.0873  & 3.1841 & 161.0389 &  2.4919 & 0.9029& 13.0412 \\
		$1/10^5$ & 0.8510 & 651.7536  & 4.1945 & 645.0605 & 4.1519 & 649.8928 &  4.1828 &0.9537 & 28.0834\\
		\noalign{\smallskip}\svhline\noalign{\smallskip}
	\end{tabular}
\end{table}
\MG{the numerically optimized parameter values when the overlap
  $\delta$ becomes small. We observe that again the optimized
  parameters behave like in Theorem \ref{Theorem2sub} and Theorem
  \ref{Theorem3sub} when the overlap $\delta$ becomes small. It is in
  principle possible to continue the asymptotic analysis from two and
  three subdomains, but this is beyond the scope of the present paper.
  Continuing the numerical optimization for five and six subdomains, we
get the results in Table \ref{Table:5sub} and Table \ref{Table:6sub},}
\begin{table}[t]
  \tabcolsep0.1em
	\centering
	\caption{Asymptotic results for five subdomains : $\sigma = \varepsilon = 1, L=1, p_a = p_b = 1$}
	\label{Table:5sub}
	\begin{tabular}{c|ccccccccc|cc}
	\hline\noalign{\smallskip}
		& \multicolumn{9}{c|}{Many parameters} &  \multicolumn{2}{c}{One parameter}  \\
		\noalign{\smallskip}\svhline\noalign{\smallskip}
		$\delta$ & $\rho$ & $p_1^+$ & $p_2^-$ & $p_2^+$ & $p_3^-$ & $p_3^+$ & $p_4^-$ & $p_4^+$ & $p_5^-$ &  $\rho$ & $p$ \\
		\noalign{\smallskip}\svhline\noalign{\smallskip}
		$1/10^2$ & 0.5273 & 8.5648  &  1.4619 &   9.1763 &   0.8030  &  9.1398  &  0.8426 &  15.5121&    2.2499 & 0.6290 & 2.6747 \\
		$1/10^3$ & 0.7333 & 24.6097  &   0.9209  & 23.4189  &  0.4499  & 37.2200  &  0.8433  & 34.8142 &   0.9181 & 0.8072 & 5.7261\\
		$1/10^4$ & 0.7769 & 156.0648  &  2.4223 & 156.0502  &  2.4221 & 161.2036  &  2.5009  &166.3478  &  2.5941 & 0.9055 & 12.3166 \\
		$1/10^5$ & 0.8547  & 704.4063  & 4.3378 & 611.3217  &  3.7296 & 611.3217   & 3.7296  & 690.8837   & 4.2116 & 0.9550& 26.5260\\
		\noalign{\smallskip}\svhline\noalign{\smallskip}
	\end{tabular}
\end{table}
\begin{table}[t]
  \tabcolsep0.1em
	\centering
	\caption{Asymptotic results for six subdomains: $\sigma = \varepsilon = 1, L=1, p_a = p_b = 1$}
	\label{Table:6sub}
	\begin{tabular}{c|ccccccccccc}
		\noalign{\smallskip}\svhline\noalign{\smallskip}
		$\delta$ & $\rho$ & $p_1^+$ & $p_2^-$ & $p_2^+$ & $p_3^-$ & $p_3^+$ & $p_4^-$ & $p_4^+$ & $p_5^-$ &  $p_5^+$ & $p_6^-$ \\
		\noalign{\smallskip}\svhline\noalign{\smallskip}
		$1/10^2$ & 0.5460 & 10.5283  &  1.4526  &  7.7653   & 1.2124 &   8.2834  &  0.6573  &  7.6445  &  1.3410 & 8.0029   & 0.9586 \\
		$1/10^3$ & 0.7011 &  30.3314  &  0.9049  & 30.3452   & 1.1096  & 30.3010  &  0.9363  & 30.3458   & 0.8901 &  30.1139  &1.1307 \\
		$1/10^4$ & 0.7837 & 145.7147  &  2.1126 & 146.4533  &  2.1231  & 145.7147  &  2.1126  &149.1802  &  2.1743 &  146.7200 & 2.1909\\
		$1/10^5$ & 0.8553 &  660.5326  &  3.9932 & 611.9401  & 3.7012 &  606.1453   & 3.6661 & 606.1144  &  3.6659  & 606.0914 & 3.8534 \\
		\noalign{\smallskip}\svhline\noalign{\smallskip}
	\end{tabular}
\end{table}
\MG{which show again the same asymptotic behavior. We therefore
  conjecture the following two results for an arbitrary fixed number
  of subdomains:}
\begin{enumerate}
  \item When all parameters are equal to $p$, then the asymptotically
    optimized parameter $p$ for small overlap $\delta$ and the
    \MG{associated} convergence factor have the same form as for two-subdomains
    \eqref{eq:p2dom} \MG{in Theorem \ref{Theorem2sub}, only the constant
    is different.}
\item If all parameters \MG{are allowed to be different, the optimized
  parameters behave for small overlap $\delta$} like
  $$
     p_j^+, \, p_{j+1}^-\in \{ 2^{-2/5} C^{2/5} \delta^{-3/5}, \,2^{-4/5} C^{4/5} \delta^{-1/5}\} \mbox{ and } p_j^+ \ne p_{j+1}^- \, \forall  j=1..,J-1,
  $$
  \MG{as we have seen in the three subdomain case in Theorem \ref{Theorem3sub},
  and we have again the same asymptotic convergence factor as for two
  and three subdomains, only the constant is
    different.}
\end{enumerate}

\section{Optimization for many subdomains}

In order to obtain a \MG{theoretical} result for many subdomains, we
use the technique of limiting spectra \cite{Bootland:2021:APS} to
\MG{derive} a bound on the spectral radius which we can then
minimize. The technique of limiting spectra allows us
  to get an estimate of the spectral radius when the matrix size goes
  to infinity.  \MG{To do so, we must however assume that the outer
  Robin boundary conditions use the same optimized parameter as at the
  interfaces, in order to have the Toeplitz structure needed for the
  limiting spectrum approach.}
\begin{theorem}[\MG{Many Subdomain Optimization}]
  With all Robin parameters equal, $p_j^{-}=p_j^{+}=p$, the
  convergence factor of the OSM satisfies the bound
  \[
    \rho = \mathop {\lim }\limits_{N \to  + \infty } \rho ({T_{2d}^{OS}}) \le \max \Big\{\left| {\alpha -\beta} \right|,\left| {\alpha + \beta} \right|\Big\}<1\MG{,}
  \]
  where 
  $
  \alpha  =\frac{(\lambda+p)^2 e^{\lambda
    \delta}-(\lambda-p)^2 e^{-\lambda \delta}}{(\lambda+p)^2e^{\lambda (L + \delta )}
  - (\lambda-p)^2 e^{ - \lambda (L + \delta )}}$, 
  $\beta =\frac{(\lambda - p)(\lambda + p) (e^{-\lambda L}-e^{\lambda L})}{(\lambda+p)(\lambda+p)e^{\lambda (L + \delta )} - (\lambda-p)(\lambda-p)e^{ - \lambda (L + \delta )}} 
  $.
  The asymptotically optimized parameter \MG{and associated
    convergence factor are}
    \begin{equation}\label{eq:Ndom}
      p= 2^{-1/3} \MG{C}^{2/3}\delta^{-1/3},\quad
      \rho=  1- 2\cdot 2^{1/3}\MG{C}^{1/3}\delta^{1/3}+{\cal O}(\delta^{2/3}) 
  \end{equation}
  \MG{with the constant $C:=\Re \frac{s(1-e^{-sL})}{1+e^{-sL}}$}.
  If we allow two-sided Robin parameters, $p_j^{-}=p^-$ and
  $p_j^+=p^{+}$, the OSM convergence factor satisfies \MG{the bound}
  \[
    \rho = \mathop {\lim }\limits_{N \to  + \infty } \rho ({T_{2d}^{OS}}) \le \max \Big\{\left| {\alpha -\sqrt{ \beta_{-} \beta_{+}}} \right|,\left| {\alpha + \sqrt{\beta_{-}\beta_{+}}} \right|\Big\}<1\MG{,}
  \]
  where 
  $
  \alpha = \frac{(\lambda+p^+)(\lambda+p^-) e^{\lambda
    \delta}-(\lambda-p^+) (\lambda-p^-)e^{-\lambda \delta}}{D}$, $\beta^{\pm} =\frac{(\lambda^2 - (p^{\mp})^2)(e^{-\lambda L}-e^{\lambda L})}{D}$, 
  with $D=(\lambda+p^+)(\lambda+p^-)e^{\lambda (L + \delta )} - (\lambda-p^+)(\lambda-p^-)e^{ - \lambda (L + \delta )}$.
  The asymptotically optimized parameter choice $p^- \ne p^+$ \MG{
   and the associated convergence factor are}    
  $$
    p^-,p^+ \in \left\{ \MG{C}^{2/5}\delta^{-3/5},\, \MG{C}^{4/5}\delta^{-1/5}  \right\},
   \quad \rho = 1- 2C^{1/5}\delta^{1/5} + {\cal O}(\delta^{2/5}),
   $$
   \MG{with the same constant $C:=\Re \frac{s(1-e^{-sL})}{1+e^{-sL}}$ as for one parameter}. 
\end{theorem}
\begin{proof}
  As in the case of two and three subdomains, we
  observe equioscillation by numerical optimization, and
  asymptotically that $p = C_p \delta^{-1/3}$, $\rho = 1 - C_R
\delta^{1/3} + {\cal O}(\delta^{2/3})$ and the convergence factor has
a local maximum at the point $\MG{\tilde{k}}^* = C_k
\delta^{-2/3}$. By \MG{expanding} for small $\delta$, the
\MG{derivative} $\frac{\partial \rho}{\partial k} (\MG{\tilde{k}}^*)$
\MG{needs to have a vanishing leading order term, which leads to} $C_p
= \frac{C_k^2}{2}$. \MG{Expanding the convergence factor at the
  maximum point $\MG{\tilde{k}}^*$ gives} $\rho(\tilde{k}^*)= \rho(C_k
\delta^{-2/3}) = 1 - 2 C_k \delta^{1/3} +{\cal O}(\delta^{2/3})$,
\MG{and hence} $C_R = 2C_k$.  \MG{Equating now} $\rho(0) =
\rho(\MG{\tilde{k}}^*)$ determines uniquely $C_k$ and then $C_p =
\sqrt{C_k/2}$ giving \eqref{eq:Ndom}. By following the same lines
\MG{as for two and three subdomains, we also} get the asymptotic
result in the case of two different parameters.
\end{proof}

\MG{We can therefore safely conclude that for the magnetotelluric
  approximation of Maxwell's equations, which contains the important
  Laplace and screened Laplace equation as special cases, it is
  sufficient to optimize transmission conditions for a simple two
  subdomain decomposition in order to obtain good transmission
  conditions also for the case of many subdomains, a new result that
  was not known so far.}

\bibliographystyle{spmpsci}
\bibliography{references}

\end{document}